\documentclass[12pt,psamsfonts]{article}
\usepackage{amsmath}
\usepackage{amssymb}
\usepackage[mathcal]{eucal}
\usepackage[all]{xy}

\newtheorem{deff}{Definition}[section]
\newtheorem{lemma}[deff]{Lemma}
\newtheorem{theorem}[deff]{Theorem}
\newtheorem{coro}[deff]{Corollary}
\newtheorem{prop}[deff]{Proposition}

\newtheorem{propo}[deff]{Proposition}

\newtheorem{em-example}[deff]{Example}
\newtheorem{em-def}[deff]{Definition}        
\newtheorem{em-remark}[deff]{Remark}         
\newtheorem{em-question}[deff]{Question}

\newtheorem{problem}[deff]{Problem}

\newenvironment{remark}{\begin{em-remark} \em }{\end{em-remark}}

\newenvironment{proof}{\noindent {\it Proof}.}{\QED}

\newcommand{\palabras}{\noindent{\it Keywords: }{\small Reflexive;
Precompact; Pseudocompact; Baire property; Open refinement condition;
$h$-embedded subgroup; Convergent sequence; Almost metrizable}}

\newcommand\proved{\vbox{\hrule\hbox{\vrule\vrule
width 0pt height 4pt depth3.5pt\hskip7pt\vrule}\hrule}}
\newcommand\QED{\hfill \proved \medskip}

\newcommand{\bcal}{{\mathcal B}}

\newcommand{\ncal}{\mathcal {N}}

\def\emp{\emptyset}

\def\ker{\mathop{\rm ker}}

\def\sm{\setminus}
\def\sub{\subseteq}

\def\om{\omega}

\DeclareMathOperator*{\Int}{Int}

\DeclareMathSymbol{\res}{\mathord}{AMSa}{"16}

\def\:{\nobreak \hskip .1111em\mathpunct {}\nonscript \mkern
   -\thinmuskip {:}\hskip .3333emplus.0555em\relax}
\def\imp{\Rightarrow}

\catcode`\@=12
\def\T{{\mathbb T}}

\def\Z{{\mathbb Z}}

\def\cont{\mathfrak c}

\title{Pontryagin duality in the class of precompact Abelian groups
and the Baire property
\footnotetext{The work was done while the authors were visiting CRM
in May--July, 2008. Both authors are deeply grateful to the CRM
staff for support.}}
\author{M.~Bruguera
\thanks{This author was partially supported by Ministerio de
Educaci\'on y Ciencia,
 grant MTM2009-14409-C02-01.}
 \and
M.~Tkachenko
\thanks{This author was supported by CONACyT of Mexico, grant
000000000074468.} }

\date{April, 2009; revised version: September, 2010}

\begin{document}
\maketitle

\begin{abstract}
We present a wide class of reflexive, precompact, non-compact,
Abelian topological groups $G$ determined by three requirements.
They must have the Baire property, satisfy the \textit{open
refinement condition}, and contain no infinite compact subsets. This
combination of properties guarantees that all compact subsets of the
dual group $G^\wedge$ are finite. We also show that many
(non-reflexive) precompact Abelian groups are quotients of reflexive
precompact Abelian groups. This includes all precompact almost
metrizable groups with the Baire property and their products. Finally,
given a compact Abelian group $G$ of weight $\geq 2^\om$, we
find proper dense subgroups $H_1$ and $H_2$ of $G$ such that
$H_1$ is reflexive and pseudocompact, while $H_2$ is non-reflexive
and almost metrizable.
\end{abstract}

\palabras
\medskip


\section{Introduction}
The Pontryagin--van~Kampen duality theory has been pushed outside
the realm of locally compact Abelian groups during the last 60 years.
It turned out that many topological Abelian groups were reflexive without
being locally compact \cite{Aus,Ban,GH,Ka,Kap,No,Pe,Ve}. Forty years
had passed after the emergence of the duality theory until the first
reflexive non-complete group was constructed by Y.~Komura in
\cite{Kom}. The difficulty of finding such examples was clarified
independently by L.~Au{\ss}enhofer and M.\,J.~Chasco in \cite{Aus}
and \cite{Cha}, respectively, where they proved that every metrizable
reflexive topological group must be complete.

However, none of the reflexive groups found up to 2008 was a proper
dense subgroup of a compact group. Subgroups of compact groups are
called \textit{precompact}. Since a precompact group is compact if and
only if it is complete, every reflexive precompact metrizable group is
compact by Au{\ss}enhofer--Chasco's theorem. The problem of whether
\textit{all\/} reflexive precompact groups are compact appears explicitly
in the article \cite{ChMa} by Chasco and Mart\'{\i}n-Peinador.

Very recently, a series of reflexive, \textit{pseudocompact}, non-compact
groups appeared in \cite{ACDT} and \cite{GM}. We recall that a
Tychonoff space is pseudocompact if every continuous real-valued
function defined on the space is bounded. Since, by \cite[Theorem~1.1]{CR},
every pseudocompact topological group is precompact, this solved the
problem.

Our aim is to continue the study of reflexivity in the class of precompact
groups and present an even wider class of precompact, non-compact
(hence non-complete) reflexive groups. In fact, one of the main results
in \cite{ACDT}, the reflexivity of pseudocompact groups without infinite
compact subsets (which was also independently proved in
\cite{GM}), is our Corollary~\ref{Cor:SBB}.

By the Comfort--Ross duality theorem in \cite{CRo}, every precompact
Abelian group $G$ becomes reflexive if both the dual group $G^\wedge$
and bidual group $G^{\wedge\wedge}$ carry the topology of pointwise
convergence, i.e., $G\cong(G_p^\wedge)_p^\wedge$ (for generalizations,
see also \cite{RT}). This result is a common basis for applications in
\cite{ACDT,GM} and here.

Our strategy is as follows. First, we work with precompact groups
$G$ without infinite compact subsets. This implies the coincidence
of the compact-open and pointwise convergence topologies on the dual
group $G^\wedge$. Then we find conditions under which the dual group
$G^\wedge$ contains no infinite compact subsets either. This finishes
the job, implying the reflexivity of $G$ by Corollary~\ref{Cor:CR}.

To guarantee the absence of infinite compact sets in $G^\wedge$ it
suffices, by Theorem~\ref{Th:2}, to require that $G$ have the Baire
property and satisfy the \textit{Open Refinement Condition\/} (ORC,
for short) defined in Subsection~\ref{Subs} below. Hence every
precompact  Abelian group with the two properties and without
infinite compact subsets is reflexive, which is our Theorem~\ref{Th:MM}.
Since pseudocompact groups have the Baire property and satisfy
ORC, we deduce in Corollary~\ref{Cor:SBB} that every pseudocompact
Abelian group without infinite compact subsets is reflexive (see
\cite[Theorem~2.2]{ACDT} or \cite[Theorem~6.1]{GM}). The proof
of Theorem~\ref{Th:MM} makes use of the dual characterization of
ORC obtained in Lemma~\ref{Le:Mon} and saying that a precompact
group $G$ satisfies ORC if and only if the closure of every countable
set in the dual group $G_p^\wedge$ endowed with the pointwise
convergence topology remains countable.

In Section~\ref{Sec:ORP} we discuss the relations between the
Baire property and ORC in precompact groups. It is shown in
Proposition~\ref{SC} and Theorem~\ref{Prop:Pro} that the class
of precompact groups with the Baire property is closed under
taking continuous homomorphic images and forming arbitrarily
large direct products, respectively. It is worth mentioning that
the Baire property fails to be productive even in linear normed
spaces \cite{Val,MP}, so one cannot omit the precompactness
requirement in Theorem~\ref{Prop:Pro}. In Lemma~\ref{Le:Bai}
and Proposition~\ref{Le:FF} we show that the Baire property and
ORC are preserved under extensions of precompact groups
by pseudocompact groups.

Our main objective in Section~\ref{Sec:4} is to find out how wide
the class of quotients of reflexive precompact groups is. The main
result here is Theorem~\ref{Th:Qu} which says that every precompact
Abelian group with the Baire property that satisfies ORC is a quotient
of a reflexive precompact group with respect to a closed
pseudocompact subgroup. These quotients include, in particular,
all almost metrizable precompact groups with the Baire property
and, hence, all precompact second countable groups with the
Baire property. These results are complemented in
Theorem~\ref{Ex:22} which says that every compact Abelian group
$G$ of weight greater than or equal to $2^\om$ contains a proper
dense reflexive pseudocompact subgroup. Finally, in
Proposition~\ref{Prop:Prod} we prove that the product of an arbitrary
family of almost metrizable precompact groups with the Baire
property is a quotient of a reflexive precompact group. This requires
two preliminary steps, Propositions~\ref{th:Per} and~\ref{Prop:PP},
where the permanence properties of the class of precompact groups
satisfying ORC are established---this class is closed under taking
continuous homomorphic images and arbitrarily large products.

Several open problems about the reflexivity and Baire property of
precompact groups are posed in Section~\ref{Sec:Problem}.

\subsection{Notation and terminology}\label{Subs}
We use the additive notation for multiplication in the case of
commutative groups, except for the circle group $\T$. All
topological groups are assumed to be Hausdorff. A subgroup
$H$ of a group $G$ is called \textit{invariant\/} if $xHx^{-1}=H$,
for each $x\in G$. The term \lq{invariant\rq} replaces \lq{normal\rq}
to avoid ambiguity when considering topological properties of
topological groups. The kernel of a group homomorphism
$f\colon G\to H$ is $\ker{f}$.

A topological group $G$ is \textit{$\om$-narrow\/} if for every
neighborhood $U$ of the neutral element in $G$, there exists a
countable set $C\sub G$ such that $UC=G$. A group $G$ is
$\omega$-narrow iff it is topologically isomorphic to a subgroup
of a product of second countable groups (see \cite[Theorem~3.4.23]{AT}).

The Ra\u{\i}kov completion of $G$ is denoted by $\varrho{G}$.
The group $\varrho{G}$ is compact if and only if $G$ is precompact
\cite[Theorem~3.7.16]{AT}. A remarkable property of the Ra\u{\i}kov
completion is that every continuous group homomorphism
$f\colon G\to H$ admits an extension to a continuous group
homomorphism $\widetilde{f}\colon \varrho{G}\to\varrho{H}$
\cite[Corollary~3.6.17]{AT}.

Given a topological Abelian group $G$, we use $G^\wedge$ to
denote the \textit{dual\/} group of $G$ consisting of all continuous
homomorphisms $\chi\colon G\to\T$ to the circle group $\T$. The
group $G^\wedge$ is endowed with the pointwise multiplication
and carries the compact-open topology. The \textit{bidual group\/}
is $G^{\wedge\wedge}=(G^\wedge)^\wedge$.

The \textit{evaluation homomorphism}
$\alpha_G\colon G\to G^{\wedge\wedge}$ is defined by the rule
$\alpha_G(x)(\chi)=\chi(x)$, for all $x\in G$ and $\chi\in G^\wedge$.
In general, $\alpha_G$ can fail to be continuous, injective or
surjective. If however $\alpha_G$ is a topological isomorphism
of $G$ onto $G^{\wedge\wedge}$, then the group $G$ is said
to be \textit{reflexive}. Pontryagin--van~Kampen's duality theorem
states that every locally compact topological Abelian group is
reflexive.

The same underlying group $G^\wedge$ endowed with the pointwise
convergence topology on elements of $G$ will be designated as
$G_p^\wedge$. A local base at the neutral element of $G_p^\wedge$
consists of the sets
 \[
 U(g_1,\ldots,g_k)=\{\chi\in G^\wedge: \chi(g_i)\in\T_+ \mbox{ for each }
 i\leq k\},
 \]
where $\T_+=\{e^{\pi{i}x}: -1/2<x<1/2\}$ and $g_1,\ldots,g_k$ are
arbitrary elements of $G$. Reformulating Corollary~\ref{Cor:CR} below,
we can say that every precompact Abelian group is reflexive when
the dual groups $G^\wedge$ and $G^{\wedge\wedge}$ carry the
pointwise convergence topology in place of the compact-open
topology.

According to \cite[Def.~4.4]{Tk1}, a subgroup $C$ of a topological
group $G$ is \textit{$h$-embedded\/} in $G$ if \textit{every\/}
homomorphism $f\colon C\to K$ to a compact group $K$ admits an
extension to a \textit{continuous\/} homomorphism $\tilde{f}\colon G\to K$.
Hence the topology of an $h$-embedded subgroup $C$ inherited
from $G$ is finer than the maximal precompact topology of the
(abstract) group $C$. If the group $G$ is precompact, then every
$h$-embedded subgroup of $G$ carries the finest precompact
topological group topology. It was shown by S.~Hern\'andez and
S.~Macario in \cite{HM} that, for precompact groups $G$, the
property that all countable subgroups of $G$ are $h$-embedded
is \textit{dual\/} to pseudocompactness. In other words, if all
countable subgroups of $G$ are $h$-embedded, then the dual
group $G^\wedge$ is pseudocompact and, in addition, if $G$
is pseudocompact, then all countable subgroups of $G^\wedge$
are $h$-embedded.

The following new concept plays an important role here. We say
that a topological group $G$ satisfies the \textit{Open Refinement
Condition\/} (abbreviated to \textit{ORC}) if for every continuous
homomorphism $f\colon G\to H$ onto a second countable group $H$,
one can find a continuous \textit{open\/} homomorphism $\pi\colon G\to
K$ onto a second countable group $K$ and a continuous homomorphism
$g\colon K\to H$ such that $f=g\circ\pi$. It is immediate from the
definition that every second countable group satisfies ORC. Since
every continuous onto homomorphism of compact groups is open, all
compact groups satisfy ORC as well.

A subset $Y$ of a space $X$ is called \textit{meager\/} or of the
\textit{first category\/} in $X$ provided that $Y$ is the union of a
countable family of nowhere dense sets in $X$. A space $X$ has the
\textit{Baire property\/} if the intersection of any countable
family of open dense sets in $X$ is dense in $X$ or, equivalently,
every non-empty open subset of $X$ is non-meager. Following \cite{Bou},
we say that a space is Baire if it has the Baire property, and the same
applies to topological groups.

A family $\ncal$ of subsets of a space $X$ is called a \textit{network\/}
for $X$ if for every $x\in X$ and every neighborhood $U$ of $x$,
there exists $N\in\ncal$ such that $x\in N\sub U$. Every base for
a space is a network, but not vice versa. It is easy to see that every
space with a countable network is separable. For unexplained
topological terms, see \cite{Eng}.

\section{Reflexive precompact groups}\label{Sec:1}
Here we discuss in detail what kind of precompact groups $G$ are
reflexive provided that all compact subsets of $G$ are finite. Our
approach leans on the Comfort--Ross duality theorem in \cite{CRo}
(a more general form of this duality can be found in \cite{MO}).
First we introduce a small piece of terminology.

Let $G$ be an abstract Abelian group and $\mathcal{P}_G$ the family of
all precompact (not necessarily Hausdorff) topological group topologies
on $G$. Let also $Hom(G,\T)$ be the group of homomorphisms of $G$
to the circle group $\T$ and $\mathcal{H}_G$ be the family of subgroups
of $Hom(G,\T)$. For every $H\in\mathcal{H}_G$, denote by $\tau_H$
the coarsest topology on $G$ that makes continuous each $\chi\in H$.
Clearly, $(G,\tau_H)$ is a precompact topological group, so $\tau_H\in
\mathcal{P}_G$ for each $H\in\mathcal{H}_G$. Conversely, given a
topology $\tau\in \mathcal{P}_G$, let $H_\tau=(G,\tau)^\wedge\in
\mathcal{H}_G$.

\begin{theorem}[Comfort--Ross duality theorem]\label{Du-CR}
Let $G$ be an abstract Abelian group and $\Phi\colon \mathcal{H}_G
\to \mathcal{P}_G$ and $\Psi\colon \mathcal{P}_G\to \mathcal{H}_G$
be mappings defined by $\Phi(H)=\tau_H$ for each $H\in \mathcal{H}_G$
and $\Psi(\tau)=H_\tau$ for each $\tau\in\mathcal{P}_G$. Then both
$\Phi$ and $\Psi$ are bijections and $\Psi=\Phi^{-1}$.
\end{theorem}

In fact, we will only use the following corollary to Theorem~\ref{Du-CR}
whose direct proof can be found in \cite{RT}:

\begin{coro}\label{Cor:CR}
Let $G$ be a precompact topological group and $G_p^\wedge$ be
the dual group endowed with the pointwise convergence topology.
Then the canonical evaluation mapping $\alpha_G\colon G\to
(G_p^\wedge)_p^\wedge$ is a topological isomorphism.
\end{coro}

The interplay between the properties of $G$ and $G_p^\wedge$,
for precompact groups $G$, was studied in \cite{HM}.

By Corollary~\ref{Cor:CR}, to guarantee the (Pontryagin) reflexivity
of a precompact group $G$, it suffices that all compact subsets of
$G$ and of the dual group $G^\wedge$ be finite---in this case the
compact-open topology and the pointwise convergence topology
coincide.

We do the job in two steps. First, in Corollary~\ref{cor:Si}, we
describe  the groups $G$ with the property that the dual group
$G^\wedge$ does not contain non-trivial convergent sequences.
Second, imposing additional conditions on $G$, we deduce that
$G^\wedge$ does not contain infinite compact subsets (see
Theorem~\ref{Th:2}).

The following important fact is well known (see \cite{FT} or
\cite[Proposition~1.4]{CMT}):

\begin{propo}\label{propImF}
Let $\{h_n: n\in\om\}$ be a family of continuous homomorphisms of a
topological group $G$ to a topological group $H$. If the set of all
$x\in G$ such that $h_n(x)$ is a Cauchy sequence in $H$ is
non-meager in $G$, then the family $\{h_n: n\in\om\}$ is
equicontinuous.
\end{propo}

We now use Proposition~\ref{propImF} to show that the dual group of
a precompact group with the Baire property does not contain
non-trivial convergent sequences.

\begin{coro}\label{cor:Si}
Let $G$ be a precompact Baire group, and suppose that a sequence
$\xi=\{h_n: n\in\om\}\sub G^\wedge$ converges pointwise on each
element of $G$. Then the sequence $\xi$ is eventually constant.
\end{coro}

\begin{proof}
Let $K=\varrho{G}$ be the completion of $G$. Then $K$ is a compact
topological group and $G$ is a non-meager dense subgroup of $K$.
Every $h_n$ admits an extension to a character $\chi_n$ of $K$. It
now follows from Proposition~\ref{propImF} that the sequence
$\xi_K=\{\chi_n: n\in\om\}$ is equicontinuous on $K$. Hence
Arzel\'a--Ascoli's theorem implies that $\xi_K$ has compact closure
in $K^\wedge$. Since the group $K^\wedge$ is discrete, we conclude
that $\xi_K$ and $\xi$ are eventually constant.
\end{proof}

In the lemma below we present a characterization of precompact
groups satisfying ORC in dual terms.

\begin{lemma}\label{Le:Mon}
A precompact Abelian group $G$ satisfies the open refinement
condition iff for every countable set $S\sub G^\wedge$, the
closure of $S$ in $G_p^\wedge$ is countable.
\end{lemma}

\begin{proof}
\textit{Necessity.}
Clearly, $G_p^\wedge$ is a topological subgroup of $\T^G$. Let $S$
be a countable subset of $G_p^\wedge$. Denote by $f$ the diagonal
product of the characters of $S$. Then $f$ is a continuous homomorphism
of $G$ to $\T^S$. Since the weight of $\T^S$ is not greater than
$|S|\cdot\om=\om$ and $G$ satisfies ORC, we can find an open
continuous homomorphism $\pi\colon G\to K$ onto a second countable
group $K$ and a continuous homomorphism $h\colon K\to\T^S$
such that $f=h\circ\pi$. Note that the group $K$ is precompact as
a continuous homomorphic of the precompact group $G$, so
the group $\varrho{K}$ is compact.

Our choice of $K$, $\pi$, and $h$ implies that $S\sub\pi^\wedge(K_p^\wedge)$,
where $\pi^\wedge\colon K_p^\wedge\to G_p^\wedge$ is the dual
homomorphism defined by $\pi^\wedge(\xi)=\xi\circ\pi$, for each
$\xi\in K_p^\wedge$. It is easy to see that $\pi^\wedge$ is a
topological and isomorphic embedding of $K_p^\wedge$ into
$G_p^\wedge$. Since the homomorphism $\pi$ is open, the image
$\pi^\wedge(K_p^\wedge)$ coincides with the annihilator of $\ker\pi$
in $G_p^\wedge$ and hence $\pi^\wedge(K_p^\wedge)$ is closed
in $G_p^\wedge$. Therefore, the closure of $S$ in $G_p^\wedge$
is homeomorphic to a subspace of $K_p^\wedge$.

Since $K$ is a dense subgroup of the compact group $\varrho{K}$,
we can identify the abstract groups $(\varrho{K})^\wedge$ and
$K^\wedge$. Clearly, $\varrho{K}$ has a countable base, so
$|K^\wedge|\leq\om$ by \cite[Theorem~24.15]{HR} or
\cite[Corollary~9.6.7]{AT}. Therefore, the closure of $S$ in
$G^\wedge$ is countable.

\textit{Sufficiency.}
Consider a continuous homomorphism $f\colon G\to H$ onto a second
countable group $H$. There exist an open continuous homomorphism
$\pi\colon G\to K$ and a continuous isomorphism (not necessarily a
homeomorphism) $i\colon K\to H$ such that $f=i\circ\pi$. Clearly,
the group $K$ is precompact.

Since $H$ is second countable, the group $H_p^\wedge$ is countable.
Let $i^\wedge\colon H_p^\wedge\to K_p^\wedge$ be the homomorphism
dual to $i$. Since $i$ is a bijection, the subgroup $i^\wedge(H_p^\wedge)$
of $K_p^\wedge$ separates points of $K$. Applying \cite[Theorem~1.9]{CR},
we conclude that $i^\wedge(H_p^\wedge)$ is dense in $K_p^\wedge$.

It follows from the continuity of the homomorphism
$\pi^\wedge\colon K_p^\wedge\to G_p^\wedge$ that the countable group
$\pi^\wedge(i^\wedge(H_p^\wedge))=f^\wedge(H_p^\wedge)$ is dense in
$\pi^\wedge(K_p^\wedge)$. By our assumption, the closure of
$f^\wedge(H_p^\wedge)$ in $G_p^\wedge$ is countable. Hence
$\pi^\wedge(K_p^\wedge)$ is also countable. But $\pi^\wedge$ is injective,
so $K_p^\wedge$ is countable as well. Therefore, the group $K$ is second
countable. Since the homomorphism $\pi\colon G\to K$ is continuous and
open, it follows from $f=i\circ\pi$ that $G$ satisfies ORC.
\end{proof}

Here is an application of Lemma~\ref{Le:Mon} to the study of
permanence properties of precompact groups satisfying ORC:

\begin{prop}\label{th:Per}
The class of precompact topological groups satisfying the open
refinement condition is closed with respect to taking continuous
homomorphic images.
\end{prop}

\begin{proof}
Let $f\colon G\to H$ be a continuous group homomorphism of $G$
onto $H$, where $G$ is precompact and satisfies ORC. The dual
homomorphism $f^\wedge\colon H_p^\wedge\to G_p^\wedge$
is a topological isomorphism of $H_p^\wedge$ onto a subgroup of
$G_p^\wedge$. Since $G$ satisfies ORC, Lemma~\ref{Le:Mon} implies
that the closure of every countable set in $G_p^\wedge$ is countable.
Hence the groups $f^\wedge(H_p^\wedge)$ and $H_p^\wedge$ have
the same property. Applying Lemma~\ref{Le:Mon} once again, we conclude
that $H$ satisfies ORC.
\end{proof}

In the next theorem we present conditions on a precompact group $G$
under which the dual group $G^\wedge$ does not contain infinite
compact subsets.

\begin{theorem}\label{Th:2}
Let $G$ be a precompact Baire group satisfying the open refinement
condition. Then the dual group $G_p^\wedge$ does not contain
infinite compact subsets.
\end{theorem}

\begin{proof}
Suppose to the contrary that $G_p^\wedge$ contains an infinite
compact subset $F$. Let $S$ be a countable infinite subset of $F$.
By Lemma~\ref{Le:Mon}, the closure of $S$ in $G_p^\wedge$, say, $C$
is a countable infinite compact set. Hence $C$ is a non-discrete
metrizable space which must contain non-trivial convergent
sequences. The latter contradicts Corollary~\ref{cor:Si}.
\end{proof}

Here is one of the main results of the article.

\begin{theorem}\label{Th:MM}
Suppose that $G$ is a precompact Abelian group satisfying the open
refinement condition. If $G$ has the Baire property and contains no
infinite compact subsets, then $G$ is reflexive.
\end{theorem}

\begin{proof}
It follows from Theorem~\ref{Th:2} that $G_p^\wedge$ has no infinite
compact subsets. Hence the groups $G^\wedge$ and $G^{\wedge\wedge}$
carry the topology of pointwise convergence on elements of $G$ and
$G^\wedge$, respectively. The conclusion now follows from
Corollary~\ref{Cor:CR}.
\end{proof}

Our next aim is to show that the class of groups satisfying the
conditions of Theorem~\ref{Th:MM} is fairly wide. For example, it
contains all pseudocompact Abelian groups without infinite compact
subsets (see Corollary~\ref{Cor:SBB} below). To see how wide the
latter class is it suffices to refer to the following weaker version of
\cite[Theorem~5.8]{GM}: Under the Singular Cardinal Hypothesis,
$SCH$, every pseudocompact Abelian group admits another
pseudocompact Hausdorff topological group topology with no
infinite compact subsets. We recall that $SCH$ is the statement
consistent with $ZFC$ and saying that for every singular cardinal
$\kappa$, if $2^{cf(\kappa)}<\kappa$, then $\kappa^{cf(\kappa)}=\kappa^+$.
In particular, $SCH$ implies that if $\kappa>2^\omega$ is a cardinal
of countable cofinality, then $\kappa^\omega=\kappa^+$.

In fact, the class of groups satisfying the conditions of
Theorem~\ref{Th:MM} contains many precompact non-pseudocompact
groups as well---this follows from Theorem~\ref{Th:Qu} or
Proposition~\ref{Prop:Prod} given below. It is worth noting that the first
examples of reflexive precompact non-pseudocompact groups were
presented in \cite[Theorem~3.3]{ACDT}.

Since every continuous homomorphism of a pseudocompact group
onto a topological group with a countable base is open \cite{CR}, the
following fact is immediate:

\begin{prop}\label{Prop:SB}
Every pseudocompact group satisfies the open refinement condition.
\end{prop}

We can now give an alternative proof of \cite[Theorem~2.2]{ACDT}
(see also \cite[Theorem~6.1]{GM}):

\begin{coro}\label{Cor:SBB}
Every pseudocompact Abelian group without infinite compact subsets
is reflexive.
\end{coro}

\begin{proof}
Every pseudocompact space $X$ has the Baire property since it meets
each non-empty $G_\delta$-set in the Stone--\v{C}ech
compactification $\beta{X}$ of $X$ (see \cite[3.10.F]{Eng}). By
Proposition~\ref{Prop:SB}, every pseudocompact topological group
satisfies ORC. Therefore, Theorem~\ref{Th:MM} implies the required
conclusion.
\end{proof}

\section{The Baire property  and the open refinement
condition}\label{Sec:ORP}
To present a series of reflexive precompact groups in
Section~\ref{Sec:4}, we need several facts about precompact groups
with the Baire property and/or satisfying ORC. In the lemma below we
establish that the Baire property is preserved under extensions of
topological groups by second countable groups. This result complements
the fact that the product $X\times Y$ of two Baire spaces is also Baire
provided one of the factors is second countable (see \cite{Ox}).

\begin{lemma}\label{Le:Ext}
Let $K$ be a closed second countable subgroup of a topological group
$G$ (not necessarily Abelian). If the coset space $G/K$ and the
group $K$ are Baire, so is $G$.
\end{lemma}

\begin{proof}
Denote by $\pi$ the natural projection of $G$ onto the left coset
space $H=G/K$. Suppose that $\{F_k: k\in\om\}$ is a sequence of
closed nowhere dense subsets of $G$. We have to show that the
complement $G\sm F$ is dense in $G$, where $F=\bigcup_{k\in\om}
F_k$.

Let $\{U_n: n\in\om\}$ be a countable local base at the neutral
element of $K$. Given integers $n,k\in\om$, we put
\[
P_{n,k}=\{h\in H: xU_n\sub F_k\cap\pi^{-1}(h),\ \mbox{for\ some}\
x\in\pi^{-1}(h)\}.
\]
Let us show that each $P_{n,k}$ is nowhere dense in $H$. Since the
group $K$ is second countable, we can find, for every $n\in\om$, a
countable set $C_n\sub K$ such that $K=U_nC_n$. We claim that
$\pi^{-1}(P_{n,k})\sub F_kC_n$, for all $n,k\in\om$. Indeed, if
$h\in P_{n,k}$ then $xU_n\sub F_k\cap\pi^{-1}(h)$, for some
$x\in\pi^{-1}(h)$. Then $\pi(x)=h$ and we have that
\[
F_kC_n\supseteq xU_nC_n=xK=\pi^{-1}(h).
\]
This implies the inclusion $\pi^{-1}(P_{n,k})\sub F_kC_n$. Since
$F_kC_n$ is a first category set in $G$, we conclude that so are
the set $\pi^{-1}(P_{n,k})$ in $G$ and its image $P_{n,k}$ in $H$,
respectively.

Let $U$ be a non-empty open subset of $G$. Then $V=\pi(U)$ is open
in $H$ and, since $P=\bigcup_{n,k\in\om}P_{n,k}$ is a first category
set in the Baire space $H$, there exists a point $y\in V\sm P$.
Choose $x\in U$ with $\pi(x)=y$. Then $U\cap xK$ is a
non-empty open subset of $xK$. It follows from our definition of the
sets $P_{n,k}$ and $P$ and the choice of the elements $y\in V$ and
$x\in U$ that $xK\cap F_k$ is a nowhere dense subset of $xK$, for
each $k\in\om$. Indeed, otherwise one can find $x'\in xK$ and
$k,n\in\omega$ such that $x'U_n\sub F_k\cap xK$. Then $y=\pi(x)
=\pi(x')\in P_{k,n}$, thus contradicting our choice of $y$.

We conclude, therefore, that $xK\cap F$ is a first category set in $xK$.
Since $K$ is Baire and $xK\cong K$, the complement $(xK\cap U)\sm F$
is non-empty. Thus, $U\sm F\neq\emp$, which finishes the proof of the
lemma.
\end{proof}

\begin{lemma}\label{Le:HM}
The class of precompact Baire groups is closed with respect
to taking continuous homomorphic images.
\end{lemma}

\begin{proof}
Suppose that $f\colon G\to H$ is a continuous surjective homomorphism
of precompact topological groups, where $G$ is Baire. Denote by $g$
an extension of $f$ to a continuous homomorphism of $\varrho{G}$
to $\varrho{H}$. Since the groups $\varrho{G}$ and $\varrho{H}$ are
compact, $g$ is an open surjection. Suppose to the contrary that $H$
can be covered by countably many nowhere dense sets, say,
$H=\bigcup_{n\in\om}C_n$. Clearly, $G\sub\bigcup_{n\in\om}F_n$,
where $F_n=g^{-1}(C_n)$ for each $n\in\om$. Since the homomorphism
$g$ is open, the sets $F_n$ are nowhere dense in $\varrho{G}$.
However, the density of $G$ in $\varrho{G}$ implies that each
$F_n\cap G$ is nowhere dense in $G$, which contradicts the
Baire property of $G$.
\end{proof}

The following result is a kind of a \textit{reflection principle\/} for
the Baire property in the class of precompact groups. It reduces
the problem of whether a given (non-metrizable) precompact group
is Baire to the verification of whether all second countable continuous
homomorphic images of the group are Baire.

\begin{prop}\label{SC}
The following conditions are equivalent for a precompact topological
group $G$ $\mathrm{(}$not necessarily Abelian$\mathrm{)\hskip-3pt
:}$
\begin{itemize}
\item[{\rm a)}] $G$ is Baire;
\item[{\rm b)}] $G$ is non-meager in its completion $\varrho{G}$;
\item[{\rm c)}] every second countable continuous homomorphic
image $H$ of $G$ is Baire.
\end{itemize}
\end{prop}

\begin{proof}
The equivalence a)~$\Leftrightarrow$~b) is almost immediate. Indeed,
suppose that $G$ can be covered by a countable family $\{F_n:
n\in\om\}$ of nowhere dense subsets of the group $\varrho{G}$. We
can assume without loss of generality that each $F_n$ is closed in
$\varrho{G}$. Then $U_n=\varrho{G}\sm F_n$ is a dense open subset of
$\varrho{G}$ and, since $G$ is dense in $\varrho{G}$, $V_n=G\cap
U_n$ is a dense open subset of $G$, for each $n\in\om$. It follows
from our definition of the sets $V_n$ that $\bigcap_{n\in\om}V_n=
\emp$, so $G$ fails to be Baire.

Conversely, suppose that $G$ is non-meager in $\varrho{G}$, and
consider a sequence $\{V_n: n\in\om\}$ of open dense subsets of $G$.
For every $n\in\om$, choose an open set $U_n$ in $\varrho{G}$ such
that $U_n\cap G=V_n$. Then $F_n=\varrho{G}\sm U_n$ is a closed
nowhere dense subset of $\varrho{G}$, so the set $P=G\sm
\bigcup_{n\in\om}F_n$ is non-empty. We claim that $P$ is dense in
$G$. If not, there exists a non-empty open set $W$ in $G$ disjoint
from $P$. Since $G$ is precompact, we can find a finite set $C\sub
G$ such that $G=WC$. It follows from our choice of $W$ that
$W\sub\bigcup_{n\in\om}F_n$, i.e., $W$ is of the first category in
$\varrho{G}$. Hence $G=WC$ is also of the first category in
$\varrho{G}$, thus contradicting our assumption about $G$.\smallskip

The implication b)~$\imp$~c) follows from Lemma~\ref{Le:HM} and the
equivalence of a) and b).\smallskip

It remains to verify that c) implies a). To this end, it suffices to
prove that for every first category set $S$ in $G$, there exists a
continuous homomorphism $f\colon G\to H$ onto a second countable
group $H$ such that $f(S)$ is of the first category in $H$.

We call an open subset $U$ of $G$ \textit{standard\/} if one can
find a continuous homomorphism $p\colon G\to H$ onto a second
countable group $H$ and an open set $V\sub H$ such that
$U=p^{-1}(V)$. Since $G$ is precompact, the standard open sets
constitute a base for $G$. It is almost immediate from the
definition that the family of standard open sets in $G$ is closed
with respect to taking countable unions.

Let $\{F_n: n\in\om\}$ be a sequence of nowhere dense subsets of
$G$. For every $n\in\om$, let $\gamma_n$ be the family of standard
open sets $U$ in $G$ such that $U\cap F_n= \emp$. Then $U_n=
\bigcup\gamma_n$ is a dense open subset of $G$ disjoint from $F_n$.
Clearly, $G$ is a dense subgroup of the compact group $\varrho{G}$,
whence it follows that $G$ has countable cellularity. Hence, for
every $n\in\om$, the family $\gamma_n$ contains a countable
subfamily $\lambda_n$ such that $\bigcup\lambda_n$ is dense in
$U_n$. Then $W_n=\bigcup\lambda_n$ is a standard open set in $G$,
for each $n\in\om$. Let a continuous homomorphism $f_n\colon G\to
H_n$ onto a second countable group $H_n$ witness that the set $W_n$
is standard, where $n\in\om$. Taking the diagonal product of the
homomorphisms $f_n$, we obtain a continuous homomorphism $f\colon
G\to H$ onto a second countable group $H$ and a family $\{O_n:
n\in\om\}$ of open sets in $H$ such that $W_n=f^{-1}(O_n)$, for each
$n\in\om$. Since each $W_n$ is dense in $G$ and $O_n=f(W_n)$, we
conclude that the sets $O_n$ are dense $H$. It also follows from the
definition of $W_n$ and the choice of $O_n$ that $f(F_n)\cap
O_n=\emp$, i.e., $f(F_n)$ is nowhere dense in $H$. Hence
$f(\bigcup_{n\in\om}F_n)$ is a first category set in $H$, as
required.
\end{proof}

It is well known that the product of two Baire spaces can fail
to be Baire (see \cite{Co,FK}). Even the product of two linear
normed spaces with the Baire property need not have it, as
M.~Valdivia showed in \cite{Val} (see also \cite{MP}). It turns
out that in the class of precompact groups, the Baire property
becomes productive.

\begin{theorem}\label{Prop:Pro}
The class of precompact Baire groups is closed under formation
of arbitrary direct products.
\end{theorem}

\begin{proof}
Let $G=\prod_{i\in I} G_i$ be a product of precompact Baire groups.
By Proposition~\ref{SC}, it suffices to verify that for every continuous
homomorphism $f\colon G\to H$ onto a second countable group $H$,
the image $H$ is Baire.

Extend $f$ to a continuous homomorphism $\varrho{f}\colon\prod_{i\in
I}\varrho{G}_i\to\varrho{H}$ and denote by $N$ the kernel of
$\varrho{f}$. Since the groups $H$ and $\varrho{H}$ are first
countable, $N$ is a $G_\delta$-set in the compact group $\varrho{G}=
\prod_{i\in I}\varrho{G}_i$. Hence we can find, for every $i\in I$,
a closed invariant subgroup $N_i\sub\varrho{G}_i$ of type $G_\delta$
in $\varrho{G}_i$ such that $\prod_{i\in I}N_i\sub N$. Let
$\pi_i\colon\varrho{G}_i\to\varrho{G}_i/N_i$ be the quotient
homomorphism, $i\in I$. Then the quotient group $\varrho{G}_i/N_i$
is compact and metrizable, so the subgroup $K_i=\pi_i(G_i)$ of
$\varrho{G}_i/N_i$ has a countable base. It follows from
Lemma~\ref{Le:HM} that the group $K_i$ is Baire.

Let $\pi=\prod_{i\in I}\pi_i$ be the product of the homomorphisms
$\pi_i$'s. Clearly, the homomorphism $\pi\colon\varrho{G}\to
\prod_{i\in I}\varrho{G}_i/N_i$ is continuous and surjective, while
$\prod_{i\in I}N_i$ is the kernel of $\pi$. Since $\varrho{G}$ is a
compact group, the homomorphism $\pi$ is open. It follows from the
inclusion $\ker\pi\sub N=\ker\varrho{f}$ that there exists a
homomorphism $\varphi\colon\prod_{i\in I}\varrho{G}_i/N_i\to
\varrho{H}$ such that $\varrho{f}=\varphi\circ\pi$. Since $\pi$ is
open, the homomorphism $\varphi$ is continuous. It is also clear
that $\pi(G)=\prod_{i\in I}K_i$.

To finish the proof, it suffices to note that the product
$K=\prod_{i\in I} K_i$ of second-countable Baire spaces is
Baire, by \cite[Theorem~3]{Ox}. Finally, by Lemma~\ref{Le:HM},
the image $H=\varphi(K)$ is Baire as well.
\end{proof}

In some cases, Lemma~\ref{Le:Ext} remains valid for a non-metrizable
subgroup $K$ of a precompact group $G$. Again, we impose no
commutativity conditions on the groups that appear in the lemma
below.

\begin{lemma}\label{Le:Bai}
Let $K$ be a closed invariant pseudocompact subgroup of a precompact
topological group $G$. If the quotient group $G/K$ is Baire, so is $G$.
\end{lemma}

\begin{proof}
By Proposition~\ref{SC}, it suffices to verify that every second
countable continuous homomorphic image $H$ of $G$ is Baire. Let
$f\colon G\to H$ be a continuous onto homomorphism. Denote by $\pi$
the canonical projection of $G$ onto $G/K$. Then the diagonal product
of $\pi$ and $f$, say $p$ is a continuous homomorphism of $G$ to the
product group $G/K\times H$. Put $M=p(G)$. Then there exist
continuous homomorphisms $\varphi \colon M\to G/K$ and
$g\colon M\to H$ satisfying $\pi=\varphi\circ{p}$ and $f=g\circ{p}$.
\[
\xymatrix{G \ar@{>}[r]^{\pi}\ar@{>}[d]_{g}
\ar@{>}[rd]^(.35){f}|!{[r];[d]}\hole &  Q  \ar@{>}[ld]|(.6){h} \ar@{-->}[d]^{\varphi}\\
P &  H \ar@{>}[l]^{\,\,i}}
\]
Clearly, $\varphi$ and $g$ are restrictions to $M$ of the
projections of $G/K\times H$ to the first and second factor,
respectively.

It is easy to see that the kernel of $\varphi$ is a compact metrizable
group. Indeed, let $e$ be the neutral element of $G/K$. Then
$\ker\varphi=\varphi^{-1}(e)=p(\pi^{-1}(e))=p(K)$ is a pseudocompact
subgroup of $M$. Since $\varphi$ is the restriction to $M$ of the
projection of the product $G/K\times H$ to the first factor, its kernel
is contained in $\{e\}\times H\cong H$, which is metrizable. Thus
$\ker\varphi$ is compact metrizable as a pseudocompact subspace
of a metrizable space.

Notice that the homomorphism $\varphi$ is open. Indeed, if $U$ is
open in $M$, then $\varphi(U)=\pi(p^{-1}(U))$. The latter set in
open in $G/K$ since $p$ is continuous and $\pi$ is open. Therefore,
$G/K$ is a quotient group of $M$ with respect to a compact
metrizable subgroup.

Since, by the assumptions of the lemma, the group $G/K$ is
Baire, it follows from Lemma~\ref{Le:Ext} that so is $M$.
Therefore, the continuous homomorphic image $H$ of $M$
is also Baire. It remains to apply Proposition~\ref{SC}
to conclude that $G$ is Baire as well.
\end{proof}

We recall that a topological group $H$ has \textit{countable pseudocharacter\/}
if the neutral element of $H$ is a $G_\delta$-set in $H$.

\begin{lemma}\label{Le:LL}
Let $f\colon G\to H$ be a continuous homomorphism of an $\om$-narrow
topological group $G$ satisfying the open refinement condition onto
a group $H$ of countable pseudocharacter. Then $H$ has a countable
network.
\end{lemma}

\begin{proof}
The group $H$ is $\om$-narrow as a continuous homomorphic image of
the $\om$-narrow group $G$. Since, by our assumptions, the neutral
element of $H$ is a $G_\delta$-set, it follows from
\cite[Corollary~5.2.12]{AT} that there exists a continuous
isomorphism $i\colon H\to P$ onto a second countable topological
group $P$. Then $g=i\circ{f}$ is a continuous homomorphism of $G$
onto $P$. Hence we can find a continuous open homomorphism
$\pi\colon G\to Q$ onto a second countable group $Q$ and a
continuous homomorphism $h\colon Q\to P$ such that $g=h\circ\pi$.
\[
\xymatrix{ & G \ar@{>}[dr]^{f}\ar@{>}[dl]_{\pi} \ar@{>}[d]^{g} & \\
Q \ar@{>}[r]^{h} \ar@/_1.8pc/[rr]^{\varphi} & P  & H \ar@{>}[l]_{i}}
\]
\vskip5pt

Then $\varphi=i^{-1}\circ{h}$ is a homomorphism of $Q$ onto $H$.
Since the homomorphism $\pi$ is open, it follows from the equality
$\varphi\circ\pi=f$ that $\varphi$ is continuous. Therefore, the
images under $\varphi$ of the elements of a countable base for $Q$
form a countable network for $H$.
\end{proof}

Here we give some additional information on Baire groups.
It will be used in the proof of Proposition~\ref{Le:FF}.

\begin{lemma}\label{Le:BPT}
Let $G$ be a Baire topological group. If $G$ has a countable network,
then $G$ is separable metrizable.
\end{lemma}

\begin{proof}
Let $\ncal$ be a countable network for $G$. Since $G$ is regular,
the closures of the elements of $\ncal$ also form a countable
network for $G$. Hence we can assume that each element of $\ncal$
is closed in $G$. For every $F\in\ncal$, let $F^*=F\sm\Int F$. Then
$F^*$ is a closed nowhere dense subset of $G$ and, since $G$ is
Baire, the set $P=G\sm\bigcup\{F^*: F\in\ncal\}$ is not empty.
Take a point $x\in P$. We claim that the family
\[
\bcal(x)=\{\Int F: x\in F\in\ncal\}
\]
is a local base for $G$ at $x$. Indeed, take an arbitrary
neighbourhood $O$ of $x$ in $G$. Since $\ncal$ is a network for
$G$, there exists $F\in\ncal$ such that $x\in F\sub O$. It follows
from $x\in P$ that $x\in\Int F$. Clearly, $\Int F\in \bcal(x)$ and
$\Int F\sub F\sub O$, whence our claim follows.

Since topological groups are homogeneous spaces, we conclude
that $G$ is first countable. Finally, every first countable topological
group is metrizable by the Birkhoff--Kakutani theorem, while
metrizable spaces with a countable network are separable and,
hence, second countable.
\end{proof}

We need one more fact about extensions of topological groups:

\begin{prop}\label{Le:FF}
Let $K$ be a closed invariant pseudocompact subgroup of an
$\om$-narrow group $G$ such that quotient group $G/K$ is
Baire and satisfies the open refinement condition. Then $G$
satisfies the open refinement condition (and is Baire).
\end{prop}

\begin{proof}
That $G$ is Baire follows directly from Lemma~\ref{Le:Bai}.
Denote by $\pi$ the canonical projection of $G$ onto the quotient
group $G/K$ and consider an arbitrary continuous homomorphism
$f\colon G\to H$ onto a second countable group $H$. Let $H_0$ be
the same underlying group $H$ that carries the quotient topology
with respect to the homomorphism $f$; then the homomorphism $f$
considered as a mapping of $G$ to $H_0$ is denoted by $f_0$. Clearly,
$f_0$ is open and continuous. Let $i$ be the identity homomorphism
of $H_0$ onto $H$. Then $i$ is continuous and satisfies
$f=i\circ{f_0}$. To finish the proof, it suffices to verify that the
group $H_0$ has a countable base.

Put $C=f(K)$ and $C_0=f_0(K)$. Then $C$ and $C_0$ are
pseudocompact invariant subgroups of $H$ and $H_0$, respectively.
Since the group $H$ is second countable, $C$ is compact. The group
$C_0$ is also compact. Indeed, the neutral element of $H_0$ is a
$G_\delta$-set since $i\colon H_0\to H$ is a continuous isomorphism
onto a second countable group. Therefore, every pseudocompact
subspace of $H_0$ is compact and has a countable base
\cite[Proposition~3.4]{Ar0}. In particular, $C_0$ is a compact
metrizable subgroup of $H_0$. It is clear that $i(C_0)=C$ and
the restriction of $i$ to $C_0$ is a topological isomorphism of
$C_0$ onto $C$.

Let $p\colon H\to H/C$ and $p_0\colon H_0\to H_0/C_0$ be
canonical projections. Then there exists a continuous isomorphism
$j\colon H_0/C_0\to H/C$ satisfying $j\circ{p_0}= p\circ{i}$. The
group $H/C$ is second countable as the image of the second
countable group $H$ under the open continuous homomorphism
$p$. Since $j$ is one-to-one and continuous, the group $G_0/H_0$
has countable pseudocharacter.

Since the kernel $K$ of the homomorphism $\pi$ is contained in
the kernel of the homomorphism $p_0\circ{f_0}$, there exists a
homomorphism $\varphi\colon G/K\to H_0/C_0$ satisfying
$\varphi\circ\pi=p_0\circ{f_0}$.
\[
\xymatrix{G \ar@{>}[rr]^{f} \ar@{>}[dr]^{f_0} \ar@{>}[dd]_{\pi}
 &  &  H \ar@{>}[dd]^{p}\\
 & H_0 \ar@{>}[d]^{p_0} \ar@{>}[ur]^{i} & \\
G/K  \ar@{>}[r]^{\varphi} & H_0/C_0 \ar@{>}[r]^{j}  & H/C}
\]
\vskip4pt

Since $\pi$ is open and the composition $p_0\circ{f_0}$ is
continuous, we conclude that $\varphi$ is also continuous.
Further, since $p_0$ and $f_0$ are open, the homomorphisms
$p_0\circ{f_0}$ and $\varphi$ are open as well. Therefore the
group $H_0/C_0$ is the Baire as an open continuous image
of the Baire group $G/K$.

Clearly, $G/K$ is $\om$-narrow as a continuous homomorphic
image of the $\om$-narrow group $G$. By the assumptions of
the proposition, the group $G/K$ satisfies ORC, while the neutral
element of the group $H_0/C_0$ is a $G_\delta$-set. Hence, by
Lemma~\ref{Le:LL}, $H_0/C_0$ has a countable network.
Then, according to Lemma~\ref{Le:BPT}, $H_0/C_0$ has
a countable base. Finally, since $C_0=\ker{p_0}$ has a
countable base, Vilenkin's theorem (see \cite{Vi} or
\cite[Corollary~1.5.21]{AT}) implies that the group $H_0$ has a
countable base as well.
\end{proof}

\section{Quotients of reflexive groups}\label{Sec:4}
We are finally in a position to show that the class of reflexive groups
satisfying the conditions of Theorem~\ref{Th:MM} contains many
precompact non-pseudo\-compact groups. Further, we will also show
that taking quotients destroys reflexivity quite easily (see
Corollaries~\ref{Cor:QPS} and~\ref{Cor:MGV}, as well as
Theorem~\ref{Ex:No}).

\begin{theorem}\label{Th:Qu}
Let $H$ be a precompact Abelian group with the Baire property that
satisfies ORC. Then there is a precompact group $G$ such that
$H=G/N$, where $N$ is a closed pseudocompact subgroup of $G$, and
\begin{enumerate}
\item[{\rm (i)}] $G$ is Baire;
\item[{\rm (ii)}] $G$ satisfies ORC;
\item[{\rm (iii)}] $G$ contains no infinite compact subsets;
\item[{\rm (iv)}] $G$ is reflexive;
\item[{\rm (v)}] all countable subgroups of $G$ are $h$-embedded.
\end{enumerate}
\end{theorem}

\begin{proof}
Denote by $\varrho{H}$ the completion of $H$. Then $\varrho{H}$
is a compact group. According to \cite[Theorem~5.5]{DT}, one can
find a pseudocompact Abelian group $P$ all countable subgroups
of which are $h$-embedded and a continuous open homomorphism
$\pi$ of $P$ onto $\varrho{H}$ such that the kernel $N$ of $\pi$ is
a pseudocompact subgroup of $P$. [The fact that all countable
subgroups of $P$ are $h$-embedded was explicitly verified in
the proof of \cite[Theorem~5.5]{DT} though it was not given
in the body of the theorem.] Hence we can apply
\cite[Proposition~2.1]{ACDT} to conclude that all compact
subsets of $P$ are finite.

Let $\varphi$ be the restriction of $\pi$ to the subgroup
$G=\pi^{-1}(H)$ of $P$. Then $G$ is precompact as a subgroup of
the pseudocompact (hence precompact) group $P$ and $\varphi$
is a continuous open homomorphism of $G$ onto $H$. It follows that
$H\cong G/N$. It is also clear that $\ker\varphi=N=\ker\pi$ is a closed
pseudocompact subgroup of $G$, so Lemma~\ref{Le:Bai} implies that
the group $G$ is Baire, i.e., $G$ satisfies (i). Since $G$ is a subgroup
of $P$, all compact subsets of $G$ are finite and all countable
subgroups of $G$ are $h$-embedded. This implies (iii) and (v).
By Proposition~\ref{Le:FF}, $G$ satisfies ORC, which gives (ii).
Therefore, item (iv) of the theorem, the reflexivity of $G$,
follows from Theorem~\ref{Th:MM}.
\end{proof}

Proposition~\ref{th:Per} shows that the open refinement condition
imposed upon the group $H$ in the above theorem appeared not
accidentally.

Since every second countable group satisfies ORC, the next result is
immediate from Theorem~\ref{Th:Qu}:

\begin{coro}\label{Cor:QPS}
Every second countable precompact Abelian group $H$ with the
Baire property is a quotient group of a reflexive precompact group
with respect to a closed pseudocompact subgroup.
\end{coro}

Recall that a topological group $G$ is called \textit{almost
metrizable\/} or \textit{feathered\/} if $G$ contains a compact
subgroup $K$ such that the quotient space $G/K$ is metrizable
(see \cite{Pa} or \cite[Section~4.3]{AT}). Clearly, all compact
groups and all metrizable groups are almost metrizable.

\begin{remark}\label{Rem:LR}
It is easy to see that for a precompact almost metrizable group $G$,
one can always find a compact \textit{invariant\/} subgroup $K$ of
$G$ such that the quotient group $G/K$ has a countable base. We
assume for simplicity that $G$ is Abelian (even if the conclusion is
valid in general). Suppose we have chosen a compact subgroup
$K$ of a precompact Abelian group $G$ such that the quotient group
$G/K$ is metrizable. Since the group $G/K$ is also precompact
(hence $\om$-narrow), $G/K$ has a countable base by
\cite[Proposition~3.4.5]{AT}.
\end{remark}

Here is a more general version of Corollary~\ref{Cor:QPS}.

\begin{coro}\label{Cor:MGV}
Every precompact almost metrizable Abelian group $H$ with the
Baire property is a quotient group of a reflexive precompact group
with respect to a closed pseudocompact subgroup.
\end{coro}

\begin{proof}
Let $K$ be a compact subgroup of $H$ such that the quotient group
$H/K$ has a countable base. Then Proposition~\ref{Le:FF} implies
that $H$ satisfies ORC. Hence the required conclusion follows from
Theorem~\ref{Th:Qu}.
\end{proof}

Many groups $H$ satisfying the conditions of
Corollary~\ref{Cor:QPS} or~\ref{Cor:MGV} need not be reflexive.
One can take, for example, any proper subgroup of $\T$ of countable
index, where $\T$ is the circle group with the usual topology. A
considerably wider class of such groups is presented below.

We say, following \cite{CRT}, that a dense subgroup $H$ of a
topological Abelian group $G$ \textit{determines} $G$ if
the dual groups $G^\wedge$ and $H^\wedge$ are topologically
isomorphic (under the natural restriction mapping). The theorem
below refines Corollary~2.10 of \cite{DSh}, where the authors just claim
the existence of a proper dense subgroup $H\sub G$ that determines
$G$. First, one simple lemma is in order:

\begin{lemma}\label{Le:Di}
If $H$ is a subgroup of a Baire topological group $G$ and the
index of $H$ in $G$ is countable, then $H$ is Baire.
\end{lemma}

\begin{proof}
Suppose to the contrary that there exists a meager
non-empty open set $U\sub H$. Since translations in $H$ are
homeomorphisms, the family $\mathcal{B}$ of the sets $xV$,
where $x\in H$ and $V$ is a non-empty open set in $H$ with
$V\sub U$, constitutes a base of $H$. It is clear that every
element of $\mathcal{B}$ is a meager subset of $G$. Let
$\mathcal{D}$ be a maximal disjoint subfamily of $\mathcal{B}$.
Then the open set $D=\bigcup\mathcal{D}$ is dense in $H$, so
$F=H\sm D$ is a closed nowhere dense set in $H$.
In particular, $F$ is nowhere dense in $G$.

We claim that $D$ is meager in $G$. Indeed, let $\mathcal{D}=
\{D_i: i\in I\}$. Since every $D_i$ is meager in $G$, we can find
a countable family $\{M_{i,n}: n\in\om\}$ of nowhere dense sets
in $G$ such that $D_i=\bigcup_{n\in\om} M_{i,n}$. Then the set
$M_n=\bigcup_{i\in I} M_{i,n}$ is nowhere dense in $G$ (we use
the fact that $\mathcal{D}$ is disjoint) and $D=\bigcup_{n\in\om} M_n$
is meager in $G$. Therefore, $H=D\cup F$ is also meager in $G$.

Finally, the group $G$ is covered by countably many cosets
of $H$, so $G$ is meager in itself. This contradicts the Baire
property of $G$.
\end{proof}

\begin{theorem}\label{Ex:No}
Every infinite compact Abelian group $G$ contains a proper dense
almost metrizable Baire subgroup $H$ which determines $G$ and
satisfies the open refinement condition. Therefore, $H$ fails to be
reflexive.
\end{theorem}

\begin{proof}
Our first step is to prove the existence of such a subgroup $H$ in
the case $G$ is metrizable. Then $G$ has countable weight,
$|G|=2^\om=\cont$, and there exists a countable dense
subgroup $S$ of $G$.

Denote by $D(G)$ the divisible hull of $G$ containing $G$ as an
essential subgroup (see \cite{Fu}). Then $D(G)$ is the direct sum of
countable subgroups, say, $D(G)=\oplus_{\alpha<\cont}C_\alpha$.
Since $S\sub G$ is countable, there exists a countable set $A\sub
\cont$ such that $S\sub\oplus_{\alpha\in A}C_\alpha$. Take any
countable infinite set $B\sub\cont$ disjoint from $A$ and put
$D=\cont\sm B$. Then $A\sub D$ and, hence, $S$ is a subgroup
of the group $H=G\cap\oplus_{\alpha\in D}C_\alpha$. Since $G$
is an essential subgroup of $D(G)$, the intersection $G\cap\oplus
_{\alpha\in B}C_\alpha$ is non-trivial, which in its turn implies
that $H$ is a proper subgroup of $G$. It is also clear that $H$ has
countable index in $G$. Since $S\sub H$, we conclude that $H$ is
dense in $G$. By Lemma~\ref{Le:Di}, $H$ is Baire.

The metrizable group $H$ is clearly almost metrizable. Since, in
addition, $H$ is dense in the compact metrizable group $G$, the
dual group $H^\wedge$ coincides with the discrete group $G^\wedge$
(see \cite[Theorem~2]{Cha} or \cite[Proposition~4.11]{Aus}). It follows
that $H$ determines $G$ and that the second dual
$H^{\wedge\wedge}\cong G^{\wedge\wedge}\cong G$ is
compact, while $H$ is not. Therefore, $H$ is not reflexive.

Suppose that $G$ is not metrizable. Then there exists a continuous
homomorphism $f\colon G\to K$ onto an infinite compact metrizable
group $K$. We have just proved that $K$ contains a proper dense
Baire subgroup, say, $H_0$. Let $H=f^{-1}(H_0)$. Then $H$ is a
proper subgroup of $G$ and, since the homomorphism $f$
is open, $H$ is dense in $G$. Let $\varphi$ be the restriction of
$f$ to $H$. Then $\varphi$ is open and $\ker f=\ker\varphi$, i.e.,
the homomorphism $\varphi$ has the compact kernel $N$. Hence
the group $H$ is almost metrizable, while Lemma~\ref{Le:Bai} implies
that $H$ is Baire. Since $N$ is compact and the quotient group
$H_0\cong H/N$ is not reflexive, neither is $H$ (see
\cite[Theorem~2.6]{BCM}). Note that $H$ satisfies ORC according to
Proposition~\ref{Prop:SB}.

Finally, since $H_0^\wedge$ is discrete and $N$ is
compact, a direct verification shows the dual group $H^\wedge$
is discrete as well. The density of $H$ in $G$ implies that the
restriction mapping $\chi\mapsto\chi\res_{H}$, with $\chi\in G^\wedge$,
is an isomorphism of the (abstract) group $G^\wedge$ onto $H^\wedge$.
Since both dual groups are discrete, this mapping is a topological
isomorphism. Hence $H$ determines $G$.
\end{proof}

Since every precompact group is a dense subgroup of a compact group,
it is natural to ask which compact groups $G$ contain proper dense
reflexive subgroups. By \cite[Theorem~2]{Cha}, the dual group of every
precompact metrizable Abelian group is discrete, so the answer is in
the negative for every compact metrizable group $G$. Therefore,
under the Continuum Hypothesis, the following result gives a
complete answer to the question.

\begin{theorem}\label{Ex:22}
Every compact Abelian group $G$ of weight greater than or equal to
$2^\om$ contains a proper dense reflexive pseudocompact subgroup.
\end{theorem}

\begin{proof}
By \cite[Theorem~24.15]{HR}, we have $|G^\wedge|=w(G)\geq\cont=
2^\om$, where $w(G)$ is the weight of $G$, and a standard argument
(along with the fact that the cofinality of $\cont$ is uncountable) shows
that $G^\wedge$ contains a subgroup $S$ isomorphic to the direct
sum of $\cont$ many non-trivial cyclic subgroups, say,
$S=\oplus_{\alpha<\cont}C_\alpha$. Since the
group $G$ is topologically isomorphic to $G^{\wedge\wedge}$, it
admits a continuous homomorphism onto the product group $\prod
_{\alpha<\cont}C_\alpha^\wedge$, where each $C_\alpha^\wedge$ is
either a finite group or the circle group $\T$. To see this, it
suffices to consider the restriction to $S$ of every character
$\chi$ defined on $G^\wedge$. Again, since $cf(\cont)>\om$, there
are either $\cont$ summands $C_\alpha$ with $C_\alpha\cong\Z$ or
$\cont$ summands $C_\alpha\cong\Z(n)$, for some integer $n>1$.
Hence $G$ admits a continuous homomorphism either onto $\T^\cont$
or onto $\Z(n)^\cont$, for some $n>1$.

According to \cite{MS}, each of the groups $\T^\cont$ and
$\Z(n)^\cont$ contains a proper dense pseudocompact subgroup
without non-trivial convergent sequences. In fact, the argument from
\cite[Theorem~5.5]{DT} shows that the corresponding subgroup $P$ can
be chosen so that all countable subgroups of $P$ are $h$-embedded.
Therefore, by \cite[Proposition~2.1]{ACDT}, all compact subsets of $P$
are finite. It now follows from Corollary~\ref{Cor:SBB} that the
group $P$ is reflexive.

Let $f\colon G\to K$ be a continuous homomorphism onto $K$, where
$K$ is either $\T^\cont$ or $\Z(n)^\cont$. Since the homomorphism
$f$ is open, $H=f^{-1}(P)$ is a proper dense pseudocompact subgroup
of $G$. The restriction $\varphi=f\res_H$ is an open continuous
homomorphism with a compact kernel $N$. Since the image $P=
\varphi(H)\cong H/N$ is a reflexive group, we conclude that so is
$H$ (see \cite[Theorem~2.6]{BCM}).
\end{proof}

It was shown in \cite[Corollary~2.10]{GG} that every connected
Abelian group $G$ of weight $\kappa=\kappa^\om$ contains a proper
dense pseudocompact subgroup without non-trivial convergent
sequences. According to Corollary~\ref{Cor:SBB}, this implies
that such a group $G$ contains a proper dense reflexive
pseudocompact subgroup. We have just proved in Theorem~\ref{Ex:22}
that $G$ does contain a proper dense reflexive pseudocompact
subgroup without the assumption that $G$ is connected.

In Proposition~\ref{Prop:PP} below we complement Proposition~\ref{th:Per}
and clarify the permanence properties of the class of topological groups
satisfying ORC, with the aim to generalize Corollary~\ref{Cor:MGV}.

\begin{prop}\label{Prop:PP}
The class of precompact topological groups satisfying the open
refinement condition is closed under the formation of arbitrary
direct products.
\end{prop}

\begin{proof}
Consider a product $G=\prod_{i\in I} G_i$, where each group $G_i$ is
precompact and satisfies ORC. Suppose that $f\colon G\to H$ is a
continuous homomorphism onto a second countable group $H$. We can
apply \cite[Lemma~8.5.4]{AT} to find a countable set $J\sub I$ and a
continuous homomorphism $q\colon\prod_{i\in J} G_i\to H$ such that
$f=q\circ\pi_J$, where $\pi_J$ is the projection of $G$ onto the
subproduct $G_J=\prod_{i\in J} G_i$. Taking into account that the
projection $\pi_J$ is an open homomorphism, we can assume without
loss of generality that the index set $I$ is countable.

Extend $f$ to a continuous homomorphism $\varrho{f}\colon
\varrho{G}\to \varrho{H}$. In what follows we identify the compact
groups $\varrho{G}$ and $\prod_{i\in I}\varrho{G}_i$. Arguing as in
the proof of Theorem~\ref{Prop:Pro}, we can find, for each $i\in
I$, a continuous homomorphism $h_i\colon \varrho{G}_i\to P_i$ onto a
second countable group $P_i$ and a continuous homomorphism
$\varphi\colon\prod_{i\in I} P_i\to \varrho{H}$ such that
$\varrho{f}=\varphi\circ{h}$, where $h\colon\prod_{i\in
I}\varrho{G}_i\to \prod_{i\in I} P_i$ is the topological product of
the homomorphisms $h_i$'s.

The subgroup $h_i(G_i)$ of $P_i$ has a countable base and, since
$G_i$ satisfies ORC, we can find an open continuous homomorphism
$g_i\colon G_i\to Q_i$ onto a second countable group $Q_i$ and a
continuous homomorphism $p_i\colon Q_i\to P_i$ satisfying
$h_i\res_{G_i}= p_i\circ{g_i}$. The topological product $g$ of the
homomorphisms $g_i$'s is an open continuous homomorphism of $G$
onto the second countable group $Q=\prod_{i\in I}Q_i$. Also we have
the equality $h\res_G=p\circ{g}$, where $p\colon Q\to P=\prod_{i\in
I} P_i$ is the product of the homomorphisms $p_i$'s.
$$
\xymatrix{G \ar@{>}[r]^{f} \ar@{>}[d]_{g} \ar@{>}[dr]^{h}
& H \\
Q \ar@{>}[r]^{p}  & P \ar@{>}[u]_{\varphi}}
$$
It follows from the definition of the homomorphisms
$\varphi,\,g,\,p\,$ that $f=(\varphi\circ{p}) \circ{g}$. Finally,
since $g\colon G\to Q$ is an open continuous homomorphism and the
group $Q$ is second countable, we conclude that $G$ satisfies ORC.
\end{proof}

We finish with a generalization of Corollary~\ref{Cor:MGV} to
topological products.

\begin{prop}\label{Prop:Prod}
Let $H$ be the product of an arbitrary family of precompact almost
metrizable Abelian groups with the Baire property. Then $H$ is a
quotient group of a reflexive precompact group $G$ with respect
to a closed pseudocompact subgroup, where all countable
subgroups of $G$ are $h$-embedded.
\end{prop}

\begin{proof}
Suppose that $H=\prod_{i\in I} H_i$, where each $H_i$ is a
precompact almost metrizable Abelian group with the Baire property.
Clearly, $H$ is precompact. It follows from Theorem~\ref{Prop:Pro}
that $H$ is Baire, while Proposition~\ref{Le:FF} and Remark~\ref{Rem:LR}
together imply that each factor $H_i$ satisfies ORC. Hence, by
Proposition~\ref{Prop:PP}, $H$ satisfies ORC as well. To finish the
proof, it suffices to apply Theorem~\ref{Th:Qu}.
\end{proof}

\begin{coro}\label{Cor:Prod}
Let $H$ be the product of an arbitrary family of precompact
metrizable Abelian groups with the Baire property. Then $H$ is
a quotient group of a reflexive precompact group with respect
to a closed pseudocompact subgroup.
\end{coro}

\section{Open problems}\label{Sec:Problem}
Theorem~\ref{Th:MM} on the reflexivity of some precompact Abelian
groups $G$ is based on the fact that both groups $G$ and its dual
$G^\wedge$ have no infinite compact subsets. However, the product
$G\times K$ of a reflexive group $G$ with any compact Abelian
group $K$ remains reflexive. Therefore, one can try to extend
Theorem~\ref{Th:Qu} to all precompact Abelian groups:

\begin{problem}\label{Prob:1}
Is every precompact Abelian group (with the Baire property) a
quotient of a reflexive precompact group (with the Baire property)?
\end{problem}

Here is a problem of a similar nature:

\begin{problem}\label{Prob:11}
Is every precompact Abelian group a continuous homomorphic image of
a reflexive precompact group $G$?
\end{problem}

The methods of the article do not work to settle the next problem
since infinite non-discrete Baire groups are uncountable:

\begin{problem}\label{Prob:VL}
Do there exist countable infinite reflexive precompact groups?
\end{problem}

It is also interesting, after Lemma~\ref{Le:Bai}, to find out
whether the Baire property is stable under extensions of groups
with respect to a compact invariant subgroup:

\begin{problem}\label{Prob:2}
Let $K$ be a compact invariant subgroup of a topological group $G$
such that the quotient group $G/K$ is Baire. Is $G$ Baire?
\end{problem}

\noindent \textbf{Acknowledgements.} The authors are indebted to
M.\,J.~Chasco for helpful discussions regarding the contents of the
article and to the referee for careful reading of the manuscript and
a number of valuable comments and suggestions, including the
sufficiency part of Lemma~\ref{Le:Mon}.


M.~Bruguera\\
{\small\em Dept. de Matem\'{a}tica Aplicada I} \\
{\small\em Universidad Polit\'{e}cnica de Catalu\~{n}a}\\
{\small\em C/ Gregorio Mara\~n\'on 44-50, 08028 Barcelona, Espa\~na}\\
{\small\em e-mail: m.montserrat.bruguera@upc.edu}
\bigskip

M.~Tkachenko\\
{\small\em Departamento de Matem\'aticas}\\
{\small\em Universidad Aut\'onoma Metropolitana}\\
{\small\em Av. San Rafael Atlixco $\#\,186$,
Col. Vicentina, Iztapalapa, }\\
{\small\em C.P. 09340,
Mexico D.\,F., Mexico}\\
{\small\em e-mail: mich@xanum.uam.mx}\\

\end{document}